\documentclass[a4paper,11pt]{article}
\usepackage{amsmath, amsfonts, amscd, amssymb, amsthm, enumerate, xypic}

\def\bfB{\mathbf{B}}
\def\bfC{\mathbf{C}}
\def\ad{\text{ad}}

\renewcommand{\epsilon}{\varepsilon}

\newcommand{\Mat}{\operatorname{M}}

\newcommand{\GL}{\operatorname{GL}}

\newcommand{\End}{\operatorname{End}}

\newcommand{\rk}{\operatorname{rk}}

\renewcommand{\setminus}{\smallsetminus}

\def\F{\mathbb{F}}

\def\R{\mathbb{R}}
\def\C{\mathbb{C}}

\def\N{\mathbb{N}}
\def\Z{\mathbb{Z}}

\def\calA{\mathcal{A}}

\def\calC{\mathcal{C}}

\def\lcro{\mathopen{[\![}}
\def\rcro{\mathclose{]\!]}}

\theoremstyle{definition}

\theoremstyle{plain}
\newtheorem{theo}{Theorem}

\newtheorem{lemma}[theo]{Lemma}

\theoremstyle{plain}

\theoremstyle{remark}
\newtheorem{Rems}{Remarks}
\newtheorem{Rem}[Rems]{Remark}

\title{The space of all $p$-th roots of a nilpotent complex matrix is path-connected}
\author{Cl\'ement de Seguins Pazzis\footnote{Universit\'e de Versailles Saint-Quentin-en-Yvelines, Laboratoire de Math\'ematiques
de Versailles, 45 avenue des Etats-Unis, 78035 Versailles cedex, France, dsp.prof@gmail.com}}

\begin{document}

\thispagestyle{plain}

\maketitle

\begin{abstract}
Let $p$ be a positive integer and $A$ be a nilpotent complex matrix. We prove that the set of all
$p$-th roots of $A$ is path-connected.
\end{abstract}

\vskip 2mm
\noindent
\emph{AMS Classification:} 15A24; 54D05.

\vskip 2mm
\noindent
\emph{Keywords:} Matrices, Nilpotency, Path-connectedness, Jordan normal form.


\section{Introduction}

Let $U$ be an open subset of the field $\C$ of complex numbers,
$f : U \rightarrow \C$ be an analytic function and $n$ be a positive integer. Given a matrix $A \in \Mat_n(\C)$,
it is natural to ask whether the matrix equation $f(X)=A$, with unknown $X \in \Mat_n(\C)$, has
at least one solution. By using the fact that $X$ commutes with $f(X)$, and by using the characteristic subspaces of $A$,
this problem can be reduced to the one of deciding whether the equation $g(X)=N$ has a solution, where
$N$ is a given \emph{nilpotent} matrix, and $g$ is a given analytic function.

There is a (not very satisfying) answer to that question, and we shall recall it in short notice. Given a nilpotent matrix $A \in \Mat_n(\C)$ and a positive integer $k$, we denote by $m_k(A)$ the number of Jordan cells of size $k$ in the Jordan normal form of $A$. The sequence
$(m_k(A))_{k \geq 1}$ is called the (Jordan) \textbf{profile} of $A$. It belongs to the additive semigroup
$\N^{(\N^*)}$ of all sequences of non-negative integers with finite support and indexed over the positive integers
(here, $\N$ denotes the set of all non-negative integers, and $\N^*$ the one of all positive integers).
More generally, any element of $\N^{(\N^*)}$ is called a \textbf{profile}. Two nilpotent matrices are similar if and only if they
have the same Jordan profile. Throughout the article, profiles will be seen as elements of the abelian group $\Z^{(\N^*)}$
of all sequences of integers with finite support.

Given $k \in \N^*$, we denote by $J_k \in \Mat_k(\C)$ the Jordan cell of size $k$
(i.e.\ the matrix of $\Mat_k(\C)$ in which the entry at the $(i,i+1)$-spot equals $1$ for all $i \in \lcro 1,k-1\rcro$, and all the other entries equal $0$), and we denote its profile by $e_k$ (so that $(e_k)_i=1$ if $i=k$, and $(e_k)_i=0$ otherwise).
We convene that $J_0$ is the $0$-by-$0$ matrix and that $e_0$ is the zero sequence in $\N^{(\N^*)}$.

The following result is folklore:

\begin{lemma}\label{powerJordan}
Let $k$ and $p$ be positive integers.
Then $J_k^p$ is similar to the direct sum of
$k-pa$ copies of $J_{a+1}$ and of $p(a+1)-k$ copies of $J_a$, for every non-negative integer $a$ such that
$pa \leq k \leq p(a+1)$ (in particular, this holds when $a$ is the quotient of $k$ modulo $p$).
\end{lemma}

From there, one proves (see Appendix A for details) that, given a nilpotent matrix $A \in \Mat_n(\C)$, the equation $f(X)=A$
has a solution if and only if the profile of $A$ belongs to the sub-semigroup of $\N^{(\N^*)}$ generated by the profiles
of the form $r\cdot e_{a+1}+(p-r)\cdot e_a$ -- where $a$ is a non-negative integer, $p$ is the finite multiplicity of some zero of $f$, and
$r \in \lcro 0,p\rcro$ -- and the profile $e_1$ if some zero of $f$ has infinite multiplicity (i.e.\ $f$ is constant on the
connected component of that zero). In particular, if $f$ has at least one simple zero then the equation
$f(X)=A$ has a solution for every nilpotent matrix $A$.

The above characterization is not very convenient though. In very special cases, one can formulate an equivalent one
that can easily be tested: a nilpotent matrix $A$ has a $p$-th root
if and only if, for all $k \in \N^*$, the integer $p-m_k(A)$ is less than or equal to the remainder of
$\underset{j=k+1}{\overset{+\infty}{\sum}} m_j(A)$ modulo $p$ \emph{provided that this remainder is non-zero}
(for example, if $p=2$ this means that $m_k(A)>0$ whenever $\underset{j=k+1}{\overset{+\infty}{\sum}} m_j(A)$ is odd).
Moreover this result holds not only over the field of complex numbers, but over any skew field.
If $f$ has exactly two zeroes, one with multiplicity $2$ and one with multiplicity $3$ (e.g. if $f : z \mapsto z^3(z-1)^2$),
then, given a nilpotent matrix $A \in \Mat_n(\C)$, the equation $f(X)=A$ has a solution if and only if
there is no pair $(k,l)$ of positive integers for which $m_k(A)=m_{k+2l}(A)=0$ and $m_{k+i}(A)=1$ for all $i \in \lcro 1,2l-1\rcro$.
We leave these results as exercises for the reader.

\vskip 3mm
Here, we will stick to the equation $X^p=A$ for a fixed nilpotent complex matrix $A$ and a fixed positive integer $p$.
When this equation has a solution, we are interested in the topological structure of its solution set
$A^{1/p}$, i.e.\ the set of all $p$-th roots of $A$. Note that all the matrices in $A^{1/p}$ are nilpotent.

A very ambitious goal is to understand the homotopy type of $A^{1/p}$. As a first step towards that goal,
we will consider here its path-connectedness.
Here is our main theorem:

\begin{theo}\label{maintheo}
Let $p$ be a positive integer and $A$ be a nilpotent complex matrix. Then the set $A^{1/p}$
is path-connected.
\end{theo}

The case $p=1$ is straightforward.
In the remainder of this section, we fix an integer $p>1$ and a nilpotent matrix $A \in \Mat_n(\C)$.
Given $m \in \N^{(\N^*)}$, we denote by $A^{1/p}_m$ the subset of all $N \in A^{1/p}$ with profile $m$
(of course this subset may be empty). We denote by $P_p(A)$ the set of all profiles $m$ such that
$A^{1/p}_m$ is non-empty.
Hence, the family $(A^{1/p}_m)_{m \in P_p(A)}$ yields a partition of $A^{1/p}$.

Two profiles $m$ and $m'$ are called \textbf{$p$-adjacent}, and we write $m \underset{p}{\sim} m'$, when
there exist non-negative integers $a,k,l$ such that $pa \leq k<l \leq p(a+1)$
and
$$m-m'=\pm(e_k+e_l-e_{k+1}-e_{l-1}).$$
Finally, we denote by $\calA_p(A)$ the set of all pairs $\{m,m'\}$ of distinct $p$-adjacent elements of $P_p(A)$.
Thus, we have defined a non-oriented graph $(P_p(A),\calA_p(A))$.

The definition of $p$-adjacency is motivated by the following basic result:

\begin{lemma}\label{Jordanppowersimilarity}
Let $a,k,l$ be integers such that $0 \leq pa \leq k<l \leq p(a+1)$.
Then the matrices $(J_k \oplus J_l)^p$ and $(J_{k+1} \oplus J_{l-1})^p$ are similar.
\end{lemma}

\begin{proof}
Denote respectively by $r$ and $s$ the remainders of $k$ and $l-1$ modulo $p$.
By Lemma \ref{powerJordan}, we find that $(J_k \oplus J_l)^p$ is similar to the direct sum of $r+(s+1)$ copies of $J_{a+1}$ and of
$(p-r)+(p-s-1)$ copies of $J_a$.
Likewise, $(J_{k+1} \oplus J_{l-1})^p$ is similar to the direct sum of $(r+1)+s$ copies of $J_{a+1}$ and of
$(p-r-1)+(p-s)$ copies of $J_a$. The claimed result ensues.
\end{proof}

We are now able to state the three steps of our proof of Theorem \ref{maintheo}:

\begin{lemma}\label{celllemma}
Let $m \in P_p(A)$. Then the space $A^{1/p}_m$ is path-connected.
\end{lemma}

\begin{lemma}\label{connectingcelllemma}
Let $m,m'$ be adjacent profiles in $P_p(A)$. Then there exist
$N \in A^{1/p}_m$ and $N' \in A^{1/p}_{m'}$ together with a path from $N$ to $N'$ in $A^{1/p}$.
\end{lemma}

\begin{lemma}\label{graphlemma}
The graph $(P_p(A),\calA_p(A))$ is connected.
\end{lemma}

Combining those three results readily yields Theorem \ref{maintheo}.

\section{Proof of Theorem \ref{maintheo}}

Throughout this part, we let $A \in \Mat_n(\C)$ be a nilpotent matrix and $p$ be a positive integer.

\subsection{Proof of Lemma \ref{celllemma}}

Let $m$ belong to $P_p(A)$.
Let $X$ and $Y$ belong to $A^{1/p}_m$. The matrices $X$ and $Y$ are nilpotent with the same profile, and hence
they are similar. Thus we have some $P \in \GL_n(\C)$ such that $Y=P X P^{-1}$. Since $X^p=Y^p=A$, we obtain that
$P$ belongs to the centralizer $C(A)$ of $A$ in the algebra $\Mat_n(\C)$. As $C(A) \cap \GL_n(\C)$ is a Zariski-open subset of the complex
finite-dimensional vector space $C(A)$, it is path-connected (see Lemma 7.2 in \cite{Shafarevich}).
Choose a path $Q : t \in [0,1] \mapsto Q(t) \in C(A) \cap \GL_n(\C)$ from $I_n$ to $P$. Then, one checks that
$q : t \in [0,1] \mapsto Q(t) X Q(t)^{-1}$ is a path from $X$ to $Y$, and $q(t)^p=Q(t) A Q(t)^{-1}=A$ for all $t \in [0,1]$.
Finally, $q(t)$ is similar to $X$ for all $t \in [0,1]$, and hence its profile is $m$.
Hence, there is a path from $X$ to $Y$ in $A^{1/p}_m$. This completes the proof of Lemma \ref{celllemma}.

\subsection{Proof of Lemma \ref{connectingcelllemma}}

As we will see, the proof of Lemma \ref{connectingcelllemma} boils down to the following basic result:

\begin{lemma}\label{basicpathlemma}
Let $a,k,l$ be integers such that $0 \leq pa \leq k<l \leq p(a+1)$.
Set $N:=k+l$. Then there exists a path $\gamma : [0,1] \rightarrow \Mat_N(\C)$ such that:
\begin{enumerate}[(i)]
\item $\gamma(0)=J_k \oplus J_l$;
\item $\gamma(1)$ is similar to $J_{k+1} \oplus J_{l-1}$;
\item the mapping $t\in [0,1] \mapsto \gamma(t)^p$ is constant.
\end{enumerate}
\end{lemma}

\begin{proof}
We shall think in terms of endomorphisms of $\C^N$: denote by $u$ the endomorphism of $\C^N$ represented by
$J_k \oplus J_l$ in the standard basis $(x_k,\dots,x_1,y_l,\dots,y_1)$ of $\C^N$.
We convene that $y_j=0$ for all $j>l$, and that $x_i=0$ for all $i>k$.
Hence, $u$ maps $x_i$ to $x_{i+1}$ for all $i>0$, and it maps
$y_j$ to $y_{j+1}$ for all $j>0$.
Given $t \in [0,1]$, define $u_t$ as the endomorphism of $\C^N$ on the standard basis by
$u_t(y_1)=(1-t)y_2+tx_1$, and by mapping any other vector $z$ of that basis to $u(z)$.
Clearly, $t \in [0,1] \mapsto u_t$ is a path in the space of all endomorphisms of $\C^N$, and $u_0=u$.

Next, one sees that $u_1$ is represented by the matrix
$J_{k+1} \oplus J_{l-1}$ in the basis $(x_k,\dots,x_1,y_1,y_l,\dots,y_2)$.

Next, let $t \in (0,1)$. One checks that
$(x_k,\dots,x_1,(1-t)y_l+t x_{l-1},\dots,(1-t)y_2+tx_1,y_1)$ is a basis of $\C^N$, and the matrix of
$u_t$ in that basis is $J_k \oplus J_l$. Hence, $u_t$ is similar to $u_0$, and it follows that
$u_t^p$ is similar to $u_0^p$.
Besides, Lemma \ref{Jordanppowersimilarity} shows that $u_1^p$ is also similar to $u_0^p$.

Now, for $t \in [0,1]$, denote by $U_t$ the matrix of $u_t$ in the standard basis of $\C^N$.
It follows from the above that $t\in [0,1] \mapsto U_t$ is a path, in the space $\Mat_N(\C)$, from $J_k \oplus J_l$
to a matrix that is similar to $J_{k+1} \oplus J_{l-1}$, and that the path $t\in [0,1] \mapsto (U_t)^p$ takes its values in the similarity class
$S(U_0^p)$ of the matrix $U_0^p$.

It is folklore that the mapping $P \in \GL_N(\C) \mapsto P U_0^p P^{-1} \in S(U_0^p)$ is a fibration
(it is a principal fibre bundle whose structural group is the group of all invertible elements of the centralizer of $U_0^p$):
see Appendix B for a short elementary proof, and the combination of Theorem 1.4.3 and Proposition 1.4.6 of \cite{Brocker} and
Proposition 8.3 of \cite{Humphreys} for a more sophisticated one.
Hence, there is a path $q : [0,1] \rightarrow \GL_N(\C)$ such that
$$\forall t \in [0,1], \; U_t^p=q(t)\, U_0^p\, q(t)^{-1}\quad \text{and} \quad q(0)=I_N.$$
Finally, we consider the path $\gamma : t \in [0,1] \mapsto q(t)^{-1}\, U_t\, q(t) \in \Mat_N(\C)$.
The above properties of $q$ show that $t \mapsto \gamma(t)^p$ is constant. Next, $\gamma(0)=U_0=J_k \oplus J_l$.
Finally, $\gamma(1)$ is similar to $U_1$ and hence to $J_{k+1} \oplus J_{l-1}$.
\end{proof}

Now, we can prove Lemma \ref{connectingcelllemma}.
Let $m,m'$ be distinct adjacent profiles in $P_p(A)$. We wish to prove that some element of $A^{1/p}_m$
is path-connected in $A^{1/p}$ to some element of $A^{1/p}_{m'}$.
Without loss of generality, we can assume that there is a non-negative integer $a$ together with
elements $k<l$ of $\lcro pa,p(a+1)\rcro$ such that $m-m'=e_k+e_l-e_{k+1}-e_{l-1}$.
As $m \neq m'$, we must have $l>k+1$, and it follows that $m_k>0$ and $m_l>0$.
Let us choose $N \in A^{1/p}_m$.
Then $N$ has at least one Jordan cell of each size $k$ and $l$. Hence, $N=P(B \oplus J_k \oplus J_l)P^{-1}$
for some nilpotent matrix $B$ and some $P \in \GL_n(\C)$. The profile of $B$ is obviously $m-e_k-e_l$.

Let us take a path $\gamma$ that satisfies the conclusion of Lemma \ref{basicpathlemma} for the pair $(k,l)$:
then, $q : t \in [0,1] \mapsto P (B \oplus \gamma(t))P^{-1}$ is a path in $\Mat_n(\C)$, and we see from
condition (iii) in Lemma \ref{basicpathlemma} that $t \mapsto q(t)^p$ is constant with value $q(0)^p=N^p=A$.
In other words, $q$ is a path in $A^{1/p}$. Finally, $q(1)$ is similar to $B \oplus \gamma(1)$, and hence to
$B \oplus J_{k+1} \oplus J_{l-1}$, whose profile equals $(m-e_k-e_l)+e_{k+1}+e_{l-1}=m'$.
Hence, $q(1) \in A^{1/p}_{m'}$. This completes the proof of Lemma \ref{connectingcelllemma}.

\subsection{Proof of Lemma \ref{graphlemma}}

We start with some preliminary notation.
Given an element $m \in \Z^{(\N^*)}$, we set
$$S(m):=\sum_{k=1}^{+\infty} k m_k$$
(called the \textbf{size} of $m$), and
$$m^{[p]}:=\biggl(\sum_{-p<k<p} (p-|k|)\,m_{pa+k}\biggr)_{a \geq 1},$$
which is an element of $\Z^{(\N^*)}$. Note that both maps
$S : \Z^{(\N^*)} \rightarrow \Z$ and $m \in \Z^{(\N^*)} \mapsto m^{[p]} \in \Z^{(\N^*)}$
are group homomorphisms.

Using the results recalled in the introduction, one sees that if $m$
is the profile of some nilpotent matrix $N$, then $m^{[p]}$ is the profile of $N^p$, while
$S(m)$ is obviously the number of rows of $N$, and hence $S(m^{[p]})=S(m)$.
Besides, using Lemma \ref{Jordanppowersimilarity}, we find that
$m^{[p]}=(m')^{[p]}$ for any two $p$-adjacent profiles $m$ and $m'$.

Given profiles $m$ and $m'$, a \textbf{$p$-chain of profiles from $m$ to $m'$} is a list $(a^{(0)},\dots,a^{(N)})$ of profiles
such that $a^{(i)} \underset{p}{\sim} a^{(i+1)}$ for all $i \in \lcro 0,N-1\rcro$, and $m=a^{(0)}$ and $m'=a^{(N)}$.

From there, Lemma \ref{graphlemma} can be seen as a reformulation of the following result:

\begin{lemma}\label{lastlemma}
Let $m,m'$ be two profiles such that $m^{[p]}=(m')^{[p]}$.
Then there is a $p$-chain of profiles from $m$ to $m'$.
\end{lemma}

\begin{proof}
Note that the assumptions yield $S(m)=S(m^{[p]})=S((m')^{[p]})=S(m')$.
We will prove the result by induction on the size of $m$.

The result is obvious if $S(m)=0$: in that case both $m$ and $m'$ equal the zero sequence, and we simply take the trivial chain $(m)$.
Assume now that $S(m)>0$.

Assume first that there exists an integer $k \geq 1$ such that $m_k>0$ and $m'_k>0$.
Then $m-e_k$ and $m'-e_k$ obviously satisfy the assumptions, and their size equals $S(m)-k$. By induction, there
is a $p$-chain $(a^{(0)},\dots,a^{(N)})$ of profiles from $m-e_k$ to $m'-e_k$.
Clearly, $(a^{(0)}+e_k,\dots,a^{(N)}+e_k)$ is a $p$-chain of profiles from $m$ to~$m'$.

Hence, in the remainder of the proof we assume that $m_km'_k=0$ for all $k \geq 1$.
Denote by $q$ the greatest positive integer such that $m_q+m'_q>0$. Without loss of generality, we can assume that $m'_q>0$ (and hence $m_q=0)$.
Denote by $a$ the least (non-negative) integer such that $q \in \lcro pa,p(a+1)\rcro$, so that $q>pa$.
Hence, $m^{[p]}_{a+1}=(m')^{[p]}_{a+1} \geq q-pa$.
In particular, $m_k>0$ for some $k \in \lcro pa+1,p(a+1)\rcro$, and we consider the greatest such integer $k$. Note that $pa<k<q$. If $m_k>1$, we note that $m-2e_k+e_{k+1}+e_{k-1}$ is still a profile that is $p$-adjacent to $m$.
If $m_k=1$, then having $m_{a+1}^{[p]} \geq q-pa$ we must also have $m_l>0$ for some $l \in \lcro pa+1,k-1\rcro$, and then we note that
$m-e_k-e_l+e_{k+1}+e_{l-1}$ is a profile. In any case, we have found a profile $a^{(k+1)}$ that is
$p$-adjacent to $m$ and for which $k+1$ is the greatest integer $i$ such that $a^{(k+1)}_i>0$.
Continuing by finite induction, we create a $p$-chain $(a^{(k)},a^{(k+1)},\dots,a^{(q)})$ of
profiles from $m$ to some profile $a^{(q)}$ such that $(a^{(q)})_q>0$.
Hence $(a^{(q)})^{[p]}=\cdots=(a^{(k)})^{[p]}=m^{[p]}=(m')^{[p]}$. As $(a^{(q)})_q>0$, the first case tackled in the above
yields a $p$-chain of profiles from $a^{(q)}$ to $m'$. Linking those $p$-chains yields a $p$-chain of profiles from $m$ to~$m'$.
\end{proof}

Lemmas \ref{celllemma} to \ref{graphlemma} are now proved, and hence Theorem \ref{maintheo} is established.

\section{Further questions}

Now that Theorem \ref{maintheo} has been proved, we wish to suggest several related open problems.
First, given an analytic function $f : U \rightarrow \C$, what are the nilpotent complex matrices $A$ for which
the set of all solutions of the equation $f(X)=A$ is path-connected? More precisely, is there a simply characterization of such matrices
in terms of the profile of $A$ and the zeroes of $f$ (and their multiplicities)?

Next, given a positive integer $p$, we wonder about the homotopy type of $A^{1/p}$. For example, if $A=0$
then $A^{1/p}$ is contractible (since it is star-shaped around $0$). However, for $E:=\begin{bmatrix}
0 & 0 & 1 \\
0 & 0 & 0 \\
0 & 0 & 0
\end{bmatrix}$, one
checks that $E^{1/2}$ is the set of all matrices of the form $\begin{bmatrix}
0 & x & y \\
0 & 0 & x^{-1} \\
0 & 0 & 0
\end{bmatrix}$, a space that is homeomorphic to $(\C \setminus \{0\}) \times \C$ and hence homotopy equivalent to
the circle $S^1$ (and not contractible!). Is there a simple way to compute the homotopy type of $A^{1/p}$
as a function of $p$ and the profile of $A$? Computing the fundamental group of $A^{1/p}$ would be interesting, for a start.

There are other interesting open questions related to the real and quaternionic cases. The
set of all square roots of $E$ with \emph{real} entries is homeomorphic to $(\R \setminus \{0\}) \times \R$, and hence it has
exactly two path-connected components. Is there a sensible way to compute the number of path-connected components of
the set of all $p$-th roots of $A$ (with real entries) as a function of $p$ and of the profile of $A$? In that prospect, it is worthwhile to note that the
real equivalent of Lemmas \ref{connectingcelllemma} and \ref{graphlemma} holds (with the same proof): the only step that fails is the real equivalent of Lemma \ref{celllemma}. Nevertheless, the set of all real $p$-th roots of $A$ is a \emph{real} affine variety, and hence it has finitely many path-connected components (alternatively, one can adapt the proof of Lemma \ref{celllemma} to yield that $A^{1/p}_m$ has finitely many path-connected components,
using the fact that $\calC(A) \cap \GL_n(\R)$ is a Zariski open subset of a finite-dimensional real vector space, see \cite{BCR}, Section 2.4).
Finally, there are similar issues in the quaternionic case: in that one however we have not succeeded in finding a single example of a nilpotent quaternionic matrix $A$ and of a positive integer $p$ such that the set of all $p$-th roots of $A$ is not path-connected.

\appendix
\section*{Appendix}

\section{When does the equation $f(X)=N$ have a solution?}

Let $U$ be an open subset of $\C$ and $f : U \rightarrow \C$ be an analytic function. Let $N \in \Mat_n(\C)$
be nilpotent. We wish to characterize the existence of a solution to the equation $f(X)=N$ with unknown $X \in \Mat_n(\C)$.

\begin{lemma}\label{JPf}
Let $N \in \Mat_n(\C)$ be a Jordan cell and $x$ be a zero of $f$ with finite multiplicity $p$.
Write $n=mp+r$ the Euclidean division of $n$ by $p$.
Then $f(x I_n+N)$ is similar to the direct sum of $r$ Jordan cells of size $m+1$ and of $p-r$ Jordan cells of size $m$.
\end{lemma}

\begin{proof}
This result is known by Lemma \ref{powerJordan} if $f : z \mapsto (z-x)^p$, in which case $f(xI_n+N)=N^p$.
In the general case we factorize $f : z \mapsto (z-x)^p g(z)$ for some analytic function $g$ on $U$. Using the commutation of $P:=g(xI_n+N)$ with $N$, we see that $N^p P$ is nilpotent and $\rk\bigl((N^p P)^k\bigr)=\rk \bigl((N^p)^k P^k\bigr)=\rk \bigl((N^p)^k\bigr)$ for every non-negative integer $k$.
Classically, the similarity class of a nilpotent matrix $M$ is characterized by the sequence of ranks $(\rk M^k)_{k \geq 0}$, and hence
$N^p P \simeq N^p$, which completes the proof.
\end{proof}

If, on the other hand, $x$ is a zero of $f$ with infinite multiplicity (i.e.\ $f$ vanishes on a whole neighborhood of $x$)
then $f(xI_n+N)=\underset{k=0}{\overset{+\infty}{\sum}} \frac{f^{(k)}(x)}{k!}\,N^k=0$ for every nilpotent matrix $N$ of $\Mat_n(\C)$,
so $f(xI_n+N)$ is the direct sum of $n$ Jordan cells of size $1$.

Now, let $X \in \Mat_n(\C)$ be such that $f(X)=N$.
The eigenvalues of $f(X)$ are the images under $f$ of those of $X$, and hence the eigenvalues of $X$ are zeroes of $f$.
Using the Jordan reduction theorem, we obtain
$$X \simeq (x_1 I_{d_1}+N_1) \oplus \cdots \oplus (x_p I_{d_p}+N_p)$$
where $x_1,\dots,x_p$ are zeroes of $f$ and $N_1,\dots,N_p$
are Jordan cells with respective positive sizes $d_1,\dots,d_p$.
Therefore
$$N=f(X) \simeq f(x_1 I_{d_1}+N_1) \oplus \cdots \oplus f(x_p I_{d_p}+N_p)$$
and it follows that the Jordan profile of $N$ is the sum of the Jordan profiles of the matrices
$f(x_k I_{d_k}+N_k)$. Using Lemma \ref{JPf} and the remark thereafter, we deduce the ``only if" part in the following statement:

\begin{theo}
Let $N \in \Mat_n(\C)$ be nilpotent. The following conditions are equivalent:
\begin{enumerate}[(i)]
\item There exists a matrix $X \in \Mat_n(\C)$ such that $f(X)=N$.
\item The Jordan profile of $N$ belongs to the sub-semigroup of $\N^{(\N^*)}$ generated by the elements of the form
$(p-r)\cdot e_a+r\cdot e_{a+1}$ where $p$ is the (finite) multiplicity of some zero of $f$, $a$ is an arbitrary non-negative integer and $r$
belongs to $\lcro 0,p\rcro$, together with the additional element $e_1$ if $f$ has a zero with infinite multiplicity.
\end{enumerate}
\end{theo}

The ``if" part of the above statement is proved in a similar fashion as the ``only if" part.

\section{The fibration $P \mapsto PAP^{-1}$}

Here, $\F$ denotes one of the fields $\R$ or $\C$.
Let $A \in \Mat_n(\F)$. Denote by $C(A)$ the centralizer of $A$
in the algebra $\Mat_n(\F)$, by $C(A)^\times$ its group of invertible elements, and by $S(A)$ the similarity class of $A$.
We wish to prove that the mapping $\pi : P \in \GL_n(\F) \mapsto PAP^{-1} \in S(A)$ defines a $C(A)^\times$-principal bundle.
For the continuous left-action $(P,M)\mapsto PMP^{-1}$ of $\GL_n(\F)$ on $\Mat_n(\F)$,
the stabilizer of $A$ is $C(A)^\times$, and hence classically it suffices to prove that the mapping
$\pi$ admits a local cross-section around $A$.

The proof is based upon the following elementary lemma:

\begin{lemma}\label{liftinglemma}
Let $V$ be a finite-dimensional vector space over $\F$. Let $u \in \End(V)$, and let $x_0 \in V$ be a non-zero vector such that $u(x_0)=0$.
Then there exists a neighborhood $U$ of $u$ in $\End(V)$, together with a continuous mapping
$f : U \rightarrow V$ such that $v[f(v)]=0$ for all $v \in U$ with the same rank as $u$, and $f(u)=x_0$.
\end{lemma}

\begin{proof}
Denote by $n$ the dimension of $V$, and by $p$ the rank of $u$.
Let us extend $x_0$ first into a basis $(e_{n-p},\dots,e_n)$ of the kernel of $u$, with $e_n=x_0$, and then into a basis
$\bfB:=(e_1,\dots,e_n)$ of $V$. We extend the linearly independent $p$-tuple $(u(e_1),\dots,u(e_p))$ into a basis
$\bfC:=(u(e_1),\dots,u(e_p),f_{p+1},\dots,f_n)$ of $V$.
In the bases $\bfB$ and $\bfC$, the matrix of $u$ reads
$$\begin{bmatrix}
I_p & 0_{p \times (n-p)} \\
0_{(n-p) \times p} & 0_{(n-p) \times (n-p)}
\end{bmatrix}.$$
For any $v \in \End(V)$, let us write its matrix in the bases $\bfB$ and $\bfC$
as
$$M(v)=\begin{bmatrix}
A(v) & C(v) \\
B(v) & D(v)
\end{bmatrix}$$
along the same pattern.
The mapping $v \in \End(V) \mapsto A(v) \in \Mat_p(\F)$ is linear, and hence continuous. It follows that
$$U:=\{v \in \End(V) : A(v)\in \GL_p(\F)\}$$
is an open subset of $\End(V)$ that contains $u$.

Next, let $v \in U$. Consider the invertible matrix
$$N(v):=\begin{bmatrix}
I_p & -A(v)^{-1}C(v) \\
0_{(n-p) \times p} & I_{n-p}
\end{bmatrix} \in \GL_n(\F),$$
so that
$M(v)N(v)=\begin{bmatrix}
A(v) & 0_{p \times (n-p)} \\
B(v) & ?
\end{bmatrix}$ has the same rank as $M(v)$.
Assume that $v$ has rank $p$. Since $A(v)$ has rank $p$, it follows that the last $n-p$ columns of $M(v)N(v)$ equal zero, and in particular
$M(v)$ annihilates the last column of $N(v)$.

For $v \in U$, denote by $f(v)$ the vector of $V$ whose matrix in $\bfB$ is the last column of
$N(v)$; obviously $f : U \rightarrow V$ is continuous, and the previous study shows that $v[f(v)]=0$
for all $v \in U$ with rank $p$. Finally, $f(u)=e_n=x_0$.
\end{proof}

\begin{Rem}
Set $p:=\rk u$ and define $\End_p(V)$ as the set of all endomorphisms of $V$ with rank $p$, and
$\xi : (u,x)\in \End_p(V) \times V \mapsto u \in \End_p(V)$ as the trivial vector bundle with fiber $V$ and base space $\End_p(V)$. The mapping
$f : (u,x) \mapsto (u,u(x))$ is obviously a $\End_p(V)$-bundle morphism from $\xi$ to itself with constant rank $p$,
therefore its kernel, which equals
$$\begin{cases}
\bigl\{(u,x) \in \End_p(V) \times V : \; u(x)=0\bigr\} & \longrightarrow \End_p(V) \\
(u,x) & \longmapsto u,
\end{cases}$$ is also a vector bundle: see \cite{Husemoller}, Chapter 3 Theorem 8.2.
The above result can then be obtained by using a local trivialization of this bundle.
\end{Rem}

We are now ready to construct the claimed local cross-section.
Consider the endomorphism $\ad_A : M \mapsto AM-MA$ of the vector space $\Mat_n(\F)$.
Denote by $p$ its rank. Applying the above lemma, we find a neighborhood $U$ of $\ad_A$ in
$\End(\Mat_n(\F))$ together with a continuous mapping $f : U \rightarrow \Mat_n(\F)$
such that $f(\ad_A)=I_n$ and $v(f(v))=0$ for all $v \in U$ with rank $p$.
The mapping
$$\Phi : B \in \Mat_n(\F) \mapsto [M \mapsto BM-MA] \in \End(\Mat_n(\F))$$
is affine, and hence continuous:
thus $U_0:=\Phi^{-1}(U)$ is a neighborhood of $A$ in $\Mat_n(\F)$.
We set
$$g : B \in U_0 \cap S(A) \mapsto f(\Phi(B))\in \Mat_n(\F),$$
so that $g(A)=I_n$.
Since $g$ is continuous, $U'_0:=g^{-1}(\GL_n(\F))$ is a neighborhood of $A$ in $S(A)$.

We will conclude the proof by showing that the restriction $g_{|U'_0}$ is a
local cross-section for the mapping $P \in \GL_n(\F) \mapsto PAP^{-1} \in S(A)$.

Let $B \in U'_0$. Since $B \in S(A)$, there is a matrix $Q \in \GL_n(\F)$ such that $B=QAQ^{-1}$.
It follows that $\Phi(B)=L_Q \circ \ad_A \circ L_Q^{-1}$ where $L_N : M \mapsto NM$ for all $N \in \Mat_n(\F)$.
Hence, $\rk \Phi(B)=\rk (\ad_A)=p$. It follows that $\Phi(B)[g(B)]=0$, that is
$Bg(B)=g(B)A$. Moreover, $g(B)$ is invertible, and hence $B=g(B)Ag(B)^{-1}$, as claimed.

\end{document}